\documentclass[submission,copyright,creativecommons]{eptcs}
\providecommand{\event}{ACT 2021}

\title{The Cost of Compositionality \\
\large A High-Performance Implementation of String Diagram Composition}
\author{Paul Wilson
\institute{University of Southampton}
\email{paul@statusfailed.com}
\and
Fabio Zanasi
\institute{University College London}
\email{f.zanasi@ucl.ac.uk}
}
\def\titlerunning{The Cost of Compositionality}
\def\authorrunning{P. Wilson \& F. Zanasi}
\date{\today}

\usepackage{amsmath, amssymb, amsthm}
\usepackage{mathrsfs} 
\usepackage{hyperref}
\usepackage{float}
\usepackage{xcolor}
\usepackage{quiver}
\usepackage{booktabs} 
\usepackage{enumitem}
\usepackage{tikzit}

\tikzstyle{edge}=[fill=white, draw=black, shape=circle]
\tikzstyle{node}=[fill=black, draw=black, shape=circle, inner sep=1.5pt]
\tikzstyle{morphism}=[fill=white, draw=black, shape=rectangle]

\tikzstyle{pointy}=[->]
\tikzstyle{dashy}=[-, dashed]
\tikzstyle{dashpoint}=[dashed, ->]
\tikzstyle{bluefill}=[-, fill={rgb,255: red,190; green,240; blue,255}]
\tikzstyle{greyfill}=[-, fill={rgb,255: red,240; green,240; blue,240}]

\usepackage{stmaryrd} 

\newtheorem{observation}{Remark}[section]

\newtheorem{definition}[observation]{Definition}

\newtheorem{remark}[observation]{Remark}

\newtheorem{proposition}[observation]{Proposition}

\newcommand{\deftext}[1]{\textbf{#1}}
\newcommand{\tr}[1]{\xrightarrow{#1}}
\newcommand{\tl}[1]{\xleftarrow{#1}}

\newcommand{\cp}[0]{\ensuremath{\fatsemi}} 
\newcommand{\id}[0]{\ensuremath{\mathsf{id}}}
\newcommand{\twist}[0]{\ensuremath{\sigma}}

\newcommand{\Cat}[1]{\ensuremath{\mathsf{#1}}}

\newcommand{\Hyp}[0]{\Cat{Hyp}}
\newcommand{\Mat}[0]{\Cat{Mat}}
\newcommand{\Mzero}[0]{\mathbf{0}}
\newcommand{\Har}[0]{\Cat{Har}}

\newcommand{\Csp}[1]{\Cat{Csp}(#1)}

\newcommand{\CspHypI}[0]{\Csp{\Cat{Hyp}_{\scriptscriptstyle \Sigma}}_{\scriptscriptstyle I}}
\newcommand{\Syn}[1]{\Cat{Free}_{\scriptscriptstyle #1}}
\newcommand{\CspHypMI}[0]{\ensuremath{\Csp{\Cat{Hyp}_{\scriptscriptstyle \Sigma}}_{\scriptscriptstyle MI}}}

\newcommand{\Bool}[0]{\ensuremath{\mathbb{B}}}
\newcommand{\Nat}[0]{\ensuremath{\mathbb{N}}}
\newcommand{\List}[0]{\ensuremath{\mathsf{List}}}

\newcommand{\hypernode}[0]{\bullet}
\newcommand{\hyperedge}[0]{\circ}
\newcommand{\inp}[1]{\mathsf{in}(#1)}
\newcommand{\out}[1]{\mathsf{out}(#1)}

\newcommand{\lbo}[1]{\ensuremath{L(#1)}}
\newcommand{\rbo}[1]{\ensuremath{R(#1)}}
\newcommand{\permeq}[1]{\ensuremath{\overset{#1}{\sim}}}
\newcommand{\ToHar}[1]{\ensuremath{\llbracket{}#1\rrbracket{}}}
\newcommand{\FromHar}[1]{\ensuremath{\langle{}#1\rangle{}}}
\newcommand{\Halfperm}[0]{\ensuremath{\mathbf{\mathtt{HALFPERM}}}}
\newcommand{\nnz}[0]{\ensuremath{\mathsf{nnz}}}

\newcommand{\gcopy}[0]{\ensuremath{\mathsf{COPY}}}
\newcommand{\gxor}[0]{\ensuremath{\mathsf{XOR}}}
\newcommand{\gand}[0]{\ensuremath{\mathsf{AND}}}
\newcommand{\gnot}[0]{\ensuremath{\mathsf{NOT}}}

\newcommand{\ObA}[0]{\ensuremath{n}}
\newcommand{\ObB}[0]{\ensuremath{m}}
\newcommand{\ObC}[0]{\ensuremath{l}}

\begin{document}
\maketitle

\begin{abstract}
String diagrams are an increasingly popular algebraic language for the analysis of graphical models of computations across different research fields. Whereas string diagrams have been thoroughly studied as semantic structures, much less attention has been given to their algorithmic properties, and efficient implementations of diagrammatic reasoning are almost an unexplored subject. 

This work intends to be a contribution in such a direction. We introduce a data structure representing string diagrams in terms of adjacency matrices. This encoding has the key advantage of providing simple
and efficient algorithms for composition and tensor product of diagrams. We demonstrate its effectiveness by showing that the complexity of the two operations is linear in the size of string diagrams. Also, as our approach is based on basic linear algebraic operations, we can take advantage of heavily optimised implementations, which we use to measure performances of string diagrammatic operations via several benchmarks.
\end{abstract}

\section{Introduction}
\label{section:introduction}



String diagrams are a ubiquitous graphical notation for depicting morphisms of a
monoidal category, and have been used in a variety of settings--- see e.g.\
\cite{coecke_kissinger_2017, genovese2021categorical, genovese2021nets, spivak2020poly} and \cite{selinger_survey_2010} for an overview.

In order to work with string diagrams on a computer, we require a
representation of them which we can manipulate.
Several such representations have been explored in the literature---see for example the
wiring diagrams of Catlab.jl~\cite{patterson2020}, and the hypergraphs of
\cite{Bonchi_2016} as used in \textsc{Cartographer}~\cite{Cartographer}.

However, to support `industrial scale' uses of string diagrams where modelisations are very large,
there is a pressing need to ensure that operations for combining these structures are \emph{efficient}.
In this work, we define a string diagram representation inspired by the parallel
programming literature (specifically \cite{MESH, practical_parallel_hypergraph}). Our data structure of choice is based on sparse adjacency matrices representing hypergraphs, and thus we call it HAR --- hypergraph adjacency representation. We shall show how to encode string diagrams into HARs, using their characterisation as hypergraphs (from \cite{Bonchi_2016}) as intermediate steps. The following picture summarises the various steps of the encoding:
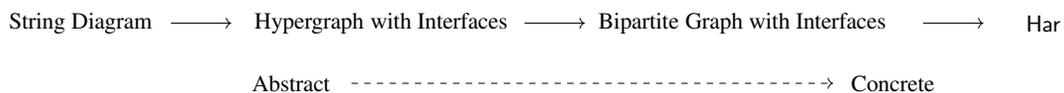
\[ \scalebox{0.8}{\begin{tikzpicture}
	\begin{pgfonlayer}{nodelayer}
		\node [style=none] (0) at (-16, 1) {String Diagram};
		\node [style=none] (1) at (-6, 1) {Hypergraph with Interfaces};
		\node [style=none] (2) at (6, 1) {Bipartite Graph with Interfaces};
		\node [style=none] (3) at (16, 1) {\Har{}};
		\node [style=none] (4) at (-9, -1) {Abstract};
		\node [style=none] (5) at (11, -1) {Concrete};
		\node [style=none] (6) at (-7, -1) {};
		\node [style=none] (7) at (9, -1) {};
		\node [style=none] (8) at (-13, 1) {};
		\node [style=none] (9) at (-11, 1) {};
		\node [style=none] (10) at (-1.25, 1) {};
		\node [style=none] (11) at (0.75, 1) {};
		\node [style=none] (12) at (12, 1) {};
		\node [style=none] (13) at (14, 1) {};
	\end{pgfonlayer}
	\begin{pgfonlayer}{edgelayer}
		\draw [style=dashpoint] (6.center) to (7.center);
		\draw [style=pointy] (8.center) to (9.center);
		\draw [style=pointy] (10.center) to (11.center);
		\draw [style=pointy] (12.center) to (13.center);
	\end{pgfonlayer}
\end{tikzpicture}} \]

The main point of this implementation is that the encoding allows for simple algorithms for composition
and tensor product.
Composition is especially simplified, being completely expressible in terms of permutation and
tensor product operations on matrices (see Definition~\ref{definition:operations} below).

Furthermore, the algorithms we describe are completely in terms of 
linear-algebraic operations on matrices.
Since highly optimised implementations of such operations are widely available,
this makes our approach straightforward to implement while providing good
performance.
Additionally, since implementations of linear algebra routines are also available for
specialised parallel hardware such as GPUs, our algorithms require
little additional effort to support such settings.

Importantly, we also show that the operations of composition and tensor product
for our representation have linear complexity, a fact which we support with
empirical validation on synthetic benchmarks.

We summarise our main contributions as follows:
\begin{itemize}
\item An isomorphic representation of string diagrams in terms of adjacency matrices of certain graphs
\item Algorithms for tensoring and composition of string diagrams (via their representation)
\item Computational complexity bounds for tensor and composition algorithms
\item An empirical analysis of the performance of our approach
\end{itemize}

The structure of the paper is as follows.
In Section \ref{section:background} we recall the directed hypergraphs of
\cite{Bonchi_2016} and discuss a bipartite encoding of \emph{undirected}
hypergraphs from the parallel processing literature~\cite{MESH}.
In Section \ref{section:representation} we formally describe our proposed
encoding of hypergraphs in terms of adjacency matrices.
We then proceed in Section \ref{section:operations} to describe operations such
as composition and tensor product for our encoding, before showing that these
form a symmetric monoidal category, which is isomorphic to the original category of string diagrams, in Section \ref{section:category}.
Finally, in Section \ref{section:complexity} we discuss the complexity of the
operations described in Section \ref{section:operations},
and show empirical performance results on some synthetic
benchmarks in Section \ref{section:empirical}.

\section{Background}
\label{section:background}


\subsection{Hypergraphs with Interfaces}

Following~\cite{Bonchi_2016}, we will regard string diagrams combinatorially as a certain class of hypergraphs, which we now recall. Throughout this section we fix a \emph{monoidal signature} $\Sigma$, that is, a set of operations $o \colon n \to m$, where $n$ is the \emph{arity} and $m$ the \emph{coarity} of $o$. 

Hypergraphs are a generalisation of directed graphs where edges (ordered pairs of vertices) are replaced by \emph{hyper}edges (ordered lists of vertices). As shown in~\cite{Bonchi_2016}, hypergraphs serve as a characterisation of string diagrams over $\Sigma$ when equipped with the following features: (i) a labeling of hyperedges with $\Sigma$-operations; (ii) the identification of a \emph{left} and a \emph{right} interface of the hypergraph; (iii) the restriction to hypergraphs with interfaces that are \emph{monogamous}. We recall the relevant definitions below.

\begin{definition}
  \label{definition:directed-multi-hypergraph}
  A \deftext{$\Sigma$-labeled (directed) hypergraph} $H$
  is a triple $(V, E, L)$,
  where
  $V$ is a set of \deftext{nodes}, 
  $E\subseteq \List(V) \times \List(V)$ is a set of \deftext{hyperedges}, and
  $L \colon E \to \Sigma$ is a \deftext{labeling} function, where the arity and coarity of $L(e)$ must agree with the length of lists $e \cp \pi_0$ and $e \cp \pi_1$ respectively, for each $e \in E$.
  A node $v$ is a \deftext{source} of $e \in E$ if it appears in the list
  $e \cp \pi_0$, and a \deftext{target} if it appears in $e \cp \pi_1$. $\Sigma$-labeled hypergraphs with the evident structure-preserving morphisms form a category $\Hyp_{\Sigma}$.

A \deftext{$\Sigma$-labeled hypergraph with interfaces} is a cospan $n \tr{f} G
  \tl{g} m$ in $\Hyp_{\Sigma}$, where $n$ and $m$ are the discrete hypergraphs
  constiting of $n$ and $m$ nodes respectively. We call $f[n]$ the \textbf{left
  interface} of $G$ and $g[m]$ the \textbf{right interface} of $G$. We write
  $\CspHypI$ for the PROP\footnote{
    PROPs~\cite{Lack2004} are symmetric monoidal categories with objects the
    natural numbers. They are widely adopted as a way to express
    algebraic theories of string diagrams categorically.
  } whose morphisms $n \to m$ are the
  hypergraph with interfaces $n \tr{f} G \tl{g} m$. Composition is defined by
  pushout.\footnote{In order for composition to be uniquely defined, strictly
  speaking morphisms of $\CspHypI$ should be equivalence classes of hypergraphs
  with interfaces, where $n \tr{} G \tl{} m$ and $n \tr{} G' \tl{} m$ are
  equivalent when there is an isomorphism $G\to G'$ commuting with the cospan
  legs. For the sake of simplicity, we shall use representatives of such
  equivalence classes when working with morphisms of $\CspHypI$. This does not
  have any consequence for the theory developed in the rest of the paper.}
  The notation is due to $\CspHypI$ being a subcategory of the category of
  \emph{cospans} in $\Hyp_{\Sigma}$.
\end{definition}
$\Sigma$-labeled hypergraphs with interfaces serve as a faithful interpretation for the PROP $\Syn{\Sigma}$ whose morphisms are freely generated by the signature $\Sigma$~\cite{Bonchi_2016}. For example, the string diagram $c \colon 2 \to 2$ in $\Syn{\Sigma}$ as on the left below is interpreted as the hypergraph with interfaces $2 \tr{f} G \tl{g} 2$  on the right.
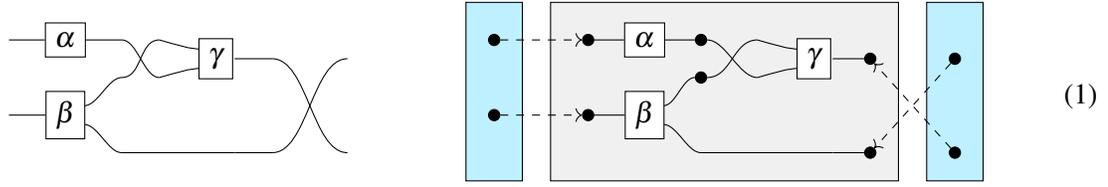
\begin{equation}
  \label{equation:string-diagram-vs-hypergraph}
  \begin{tikzpicture}
	\begin{pgfonlayer}{nodelayer}
		\node [style=none] (0) at (-4.5, 1.5) {};
		\node [style=morphism] (2) at (-3, 1.5) {$\alpha$};
		\node [style=none] (3) at (-1.5, 1.5) {};
		\node [style=none] (4) at (-4.5, -0.5) {};
		\node [style=morphism] (5) at (-3, -0.5) {$\beta$};
		\node [style=none] (6) at (-2.5, -0.25) {};
		\node [style=none] (7) at (-2.5, -0.75) {};
		\node [style=none] (8) at (-1.5, 0.5) {};
		\node [style=none] (9) at (-1.5, -1.5) {};
		\node [style=none] (10) at (-0.5, 0.5) {};
		\node [style=none] (11) at (-0.5, 1.5) {};
		\node [style=morphism] (12) at (1, 1) {$\gamma$};
		\node [style=none] (14) at (1.5, 1) {};
		\node [style=none] (15) at (0.75, 1.25) {};
		\node [style=none] (16) at (0.75, 0.75) {};
		\node [style=none] (17) at (1.5, -1.5) {};
		\node [style=none] (18) at (2.5, -1.5) {};
		\node [style=none] (19) at (2.5, 1) {};
		\node [style=none] (20) at (4.5, 1) {};
		\node [style=none] (21) at (4.5, -1.5) {};
	\end{pgfonlayer}
	\begin{pgfonlayer}{edgelayer}
		\draw (4.center) to (5);
		\draw (0.center) to (2);
		\draw (2) to (3.center);
		\draw [in=180, out=0, looseness=0.75] (7.center) to (9.center);
		\draw [in=180, out=0, looseness=0.75] (6.center) to (8.center);
		\draw [in=-180, out=0, looseness=0.75] (3.center) to (10.center);
		\draw [in=-180, out=0] (8.center) to (11.center);
		\draw [in=-180, out=-15] (11.center) to (15.center);
		\draw [in=180, out=15] (10.center) to (16.center);
		\draw (9.center) to (17.center);
		\draw (14.center) to (19.center);
		\draw (17.center) to (18.center);
		\draw [in=180, out=0, looseness=0.75] (18.center) to (20.center);
		\draw [in=-180, out=0, looseness=0.75] (19.center) to (21.center);
	\end{pgfonlayer}
\end{tikzpicture}
  \qquad \qquad
  \begin{tikzpicture}
	\begin{pgfonlayer}{nodelayer}
		\node [style=none] (24) at (-7.75, 2.5) {};
		\node [style=none] (25) at (-6.25, 2.5) {};
		\node [style=none] (26) at (-6.25, -2.25) {};
		\node [style=none] (27) at (-7.75, -2.25) {};
		\node [style=none] (28) at (4.5, 2.5) {};
		\node [style=none] (29) at (6, 2.5) {};
		\node [style=none] (30) at (6, -2.25) {};
		\node [style=none] (31) at (4.5, -2.25) {};
		\node [style=none] (32) at (-5.5, 2.5) {};
		\node [style=none] (33) at (-5.5, -2.25) {};
		\node [style=none] (34) at (3.75, -2.25) {};
		\node [style=none] (35) at (3.75, 2.5) {};
		\node [style=node] (0) at (-4.5, 1.5) {};
		\node [style=morphism] (2) at (-3, 1.5) {$\alpha$};
		\node [style=node] (3) at (-1.5, 1.5) {};
		\node [style=node] (4) at (-4.5, -0.5) {};
		\node [style=morphism] (5) at (-3, -0.5) {$\beta$};
		\node [style=none] (6) at (-2.5, -0.25) {};
		\node [style=none] (7) at (-2.5, -0.75) {};
		\node [style=node] (8) at (-1.5, 0.5) {};
		\node [style=none] (9) at (-1.5, -1.5) {};
		\node [style=none] (10) at (0, 0.5) {};
		\node [style=none] (11) at (0, 1.5) {};
		\node [style=morphism] (12) at (1.5, 1) {$\gamma$};
		\node [style=none] (14) at (2, 1) {};
		\node [style=none] (15) at (1.25, 1.25) {};
		\node [style=none] (16) at (1.25, 0.75) {};
		\node [style=none] (17) at (2, -1.5) {};
		\node [style=node] (18) at (3, -1.5) {};
		\node [style=node] (19) at (3, 1) {};
		\node [style=node] (20) at (5.25, 1) {};
		\node [style=node] (21) at (5.25, -1.5) {};
		\node [style=node] (22) at (-7, 1.5) {};
		\node [style=node] (23) at (-7, -0.5) {};
	\end{pgfonlayer}
	\begin{pgfonlayer}{edgelayer}
		\draw [style=bluefill] (25.center)
			 to (26.center)
			 to (27.center)
			 to (24.center)
			 to cycle;
		\draw [style=bluefill] (29.center)
			 to (30.center)
			 to (31.center)
			 to (28.center)
			 to cycle;
		\draw [style=greyfill] (33.center)
			 to (32.center)
			 to (35.center)
			 to (34.center)
			 to cycle;
		\draw (4) to (5);
		\draw (0) to (2);
		\draw (2) to (3);
		\draw [in=180, out=0, looseness=0.75] (7.center) to (9.center);
		\draw [in=180, out=0, looseness=0.75] (6.center) to (8);
		\draw [in=-180, out=0, looseness=0.75] (3) to (10.center);
		\draw [in=-180, out=0] (8) to (11.center);
		\draw [in=-180, out=-15] (11.center) to (15.center);
		\draw [in=180, out=15] (10.center) to (16.center);
		\draw (9.center) to (17.center);
		\draw (14.center) to (19);
		\draw (17.center) to (18);
		\draw [style=dashpoint] (21) to (19);
		\draw [style=dashpoint] (20) to (18);
		\draw [style=dashpoint] (23) to (4);
		\draw [style=dashpoint] (22) to (0);
	\end{pgfonlayer}
\end{tikzpicture}
\end{equation}
Note $\Sigma$-operations $\alpha \colon 1 \to 1$, $\beta \colon 1 \to 2$ and $\gamma \colon 2 \to 1$ appearing in $c$ are mapped to hyperedges with the appropriate number of source and target nodes, and the `dangling wires' of $c$ are expressed by the left and right interfaces $2 \tr{f} G$ and $2 \tr{g} G$, depicted as dashed arrows.
Note also that diagrams consisting of only wires are represented by hypergraphs with no hyperedges. For example, the string diagram and hypergraph representation for $\id_2$ are depicted as follows:
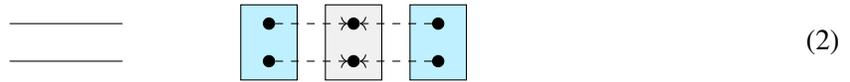
\begin{equation}
  \label{equation:identity-string-diagram-vs-hypergraph}
  \begin{tikzpicture}
	\begin{pgfonlayer}{nodelayer}
		\node [style=none] (0) at (-1.5, 0.5) {};
		\node [style=none] (4) at (-1.5, -0.5) {};
		\node [style=none] (20) at (1.5, 0.5) {};
		\node [style=none] (21) at (1.5, -0.5) {};
	\end{pgfonlayer}
	\begin{pgfonlayer}{edgelayer}
		\draw (0.center) to (20.center);
		\draw (4.center) to (21.center);
	\end{pgfonlayer}
\end{tikzpicture}

  \qquad \qquad
  \begin{tikzpicture}
	\begin{pgfonlayer}{nodelayer}
		\node [style=none] (24) at (-3, 1) {};
		\node [style=none] (25) at (-1.5, 1) {};
		\node [style=none] (26) at (-1.5, -1) {};
		\node [style=none] (27) at (-3, -1) {};
		\node [style=none] (28) at (1.5, 1) {};
		\node [style=none] (29) at (3, 1) {};
		\node [style=none] (30) at (3, -1) {};
		\node [style=none] (31) at (1.5, -1) {};
		\node [style=none] (32) at (-0.75, 1) {};
		\node [style=none] (33) at (-0.75, -1) {};
		\node [style=none] (34) at (0.75, -1) {};
		\node [style=none] (35) at (0.75, 1) {};
		\node [style=node] (18) at (0, -0.5) {};
		\node [style=node] (19) at (0, 0.5) {};
		\node [style=node] (20) at (2.25, -0.5) {};
		\node [style=node] (21) at (2.25, 0.5) {};
		\node [style=node] (22) at (-2.25, 0.5) {};
		\node [style=node] (23) at (-2.25, -0.5) {};
	\end{pgfonlayer}
	\begin{pgfonlayer}{edgelayer}
		\draw [style=bluefill] (25.center)
			 to (26.center)
			 to (27.center)
			 to (24.center)
			 to cycle;
		\draw [style=bluefill] (29.center)
			 to (30.center)
			 to (31.center)
			 to (28.center)
			 to cycle;
		\draw [style=greyfill] (33.center)
			 to (32.center)
			 to (35.center)
			 to (34.center)
			 to cycle;
		\draw [style=dashpoint] (21) to (19);
		\draw [style=dashpoint] (20) to (18);
		\draw [style=dashpoint] (22) to (19);
		\draw [style=dashpoint] (23) to (18);
	\end{pgfonlayer}
\end{tikzpicture}
\end{equation}

Although this interpretation is faithful, it is not full --- there are hypergraphs not representing any string diagram~\cite{Bonchi_2016}. One may achieve a full interpretation by restricting to \emph{monogamous acyclic} hypergraphs with interfaces.

Before recalling the definition of monogamous, we need to record a few preliminaries. The \emph{in-degree} (respectively, \emph{out-degree}) of a node $v$ in a hypergraph $G$ is the number of hyperedges having $v$ as target (source). Write $\inp{G}$ (\emph{inputs}) for the set of nodes with in-degree $0$ and $\out{G}$ (\emph{outputs}) for the set of nodes with out-degree $0$.
\begin{definition}
	  A hypergraph $G$ is \textbf{monogamous acyclic} (ma-hypergraph) if it contains no cycle (acyclicity) and every node has at most in- and out-degree $1$ (monogamy).

  A hypergraph with interfaces $n \tr{f} G \tl{g} m$ is \textbf{monogamous acyclic}  when $G$ is an ma-hypergraph, $f$ is a monomorphism and its image is $\inp{G}$, $g$ is a monomorphism and its image is $\out{G}$. Ma-hypergraphs with interfaces form a sub-PROP $\CspHypMI$ of $\CspHypI$.
\end{definition}

Ma-hypergraphs with interfaces are in 1-to-1 correspondence with string diagrams over the same signature, yielding the isomorphism $\Syn{\Sigma} \cong \CspHypMI$~\cite{Bonchi_2016}.



\subsection{Parallel Hypergraph Processing}\label{section:parahyprpoc}

Hypergraphs have many choices of implementation as a data structure. For example,
one might choose to model hyperedges directly as pairs of lists.
However, the code for such representations must be written from scratch, and can
typically be complicated and error-prone.
Instead, we would like to take advantage of existing,
high performance code for representing graphs, and apply it to hypergraphs.  To
this end, we take inspiration from the parallel programming literature.


The authors of \cite{MESH} describe a distributed processing system for
\emph{undirected} hypergraphs.

\begin{definition}
  \label{definition:undirected-hypergraph}
  An \deftext{undirected hypergraph} $U$ is a pair $(V, E)$ where $V$ is a set of \deftext{nodes},
  and $E \subseteq \mathscr{P}(V) \setminus \emptyset$ is a set of \deftext{hyperedges}.
\end{definition}

Analogously, this definition is a generalisation of the notion of
\emph{undirected graphs}: where an edge is an \emph{unordered} pair of vertices,
a \emph{hyper}edge is a \emph{set} of vertices.

In order to achieve high performance, the authors define an encoding of
their undirected hypergraphs as labeled bipartite graphs.
Concretely, vertices are labeled either $\hypernode$ or $\hyperedge$,
with $\hypernode$-vertices playing the role of hypernodes,
and $\hyperedge$-vertices playing the role of hyperedges.
For example, the bipartite graph below depicts such an encoding, where
an edge $\hypernode \to \hyperedge$ indicates that the source hypernode
appears in the hyperedge set.

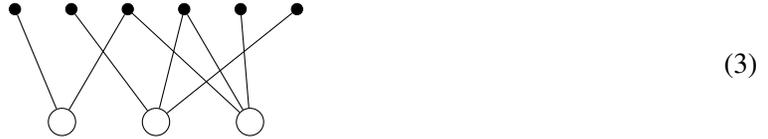
\begin{equation}
  \label{equation:bipartite-encoding-undirected-hypergraph}
  \begin{tikzpicture}
	\begin{pgfonlayer}{nodelayer}
		\node [style=node] (0) at (-3.75, 1.5) {};
		\node [style=node] (1) at (-2.25, 1.5) {};
		\node [style=node] (2) at (-0.75, 1.5) {};
		\node [style=node] (3) at (0.75, 1.5) {};
		\node [style=node] (4) at (2.25, 1.5) {};
		\node [style=node] (5) at (3.75, 1.5) {};
		\node [style=edge] (6) at (-2.5, -1.5) {};
		\node [style=edge] (7) at (0, -1.5) {};
		\node [style=edge] (8) at (2.5, -1.5) {};
	\end{pgfonlayer}
	\begin{pgfonlayer}{edgelayer}
		\draw (0) to (6);
		\draw (2) to (6);
		\draw (1) to (7);
		\draw (3) to (7);
		\draw (2) to (8);
		\draw (3) to (8);
		\draw (4) to (8);
		\draw (5) to (7);
	\end{pgfonlayer}
\end{tikzpicture}
\end{equation}

However, since this encoding is specific to the undirected hypergraphs of
\cite{MESH}, we must adapt it to suit our purposes.

\subsection{PROPs of Matrices}
\label{section:props-of-matrices}

We denote the PROP of matrices over a semiring $S$ by $\Mat_S$,
where the tensor product $f \otimes g$ is the \textbf{direct sum}, i.e.:
$ {\small \left|\begin{matrix}
    f & \Mzero \\
    \Mzero & g
  \end{matrix}\right|}
$.
In this paper, we will only consider matrices over the semirings of booleans
$\Bool$ and natural numbers $\Nat$.
We denote the $\ObB \times \ObA$ zero matrix as $\Mzero_{\ObA, \ObB}$, dropping the subscripts
where unambiguous.
We refer to the set of $\ObB \times \ObA$ matrices as $\Mat_S(\ObA, \ObB)$,
and note that we always write composition in diagrammatic order:
that is, for composition $\ObA \overset{f}{\to} \ObB \overset{g}{\to} \ObC$ we always write $f \cp g$.

\subsection{Adjacency Matrices}\label{section:adjacencymatrices}

The adjacency matrix representation of a graph is central to our representation
of hypergraphs, and so we recall it now.
The adjacency matrix of a $K$-node directed graph is a matrix $\Mat_\Bool(K, K)$ where
the $i^{\mathrm{th}}$ column denotes the \emph{outgoing} edges of the
$i^{\mathrm{th}}$ node.
For example, consider the graph and its adjacency matrix below:

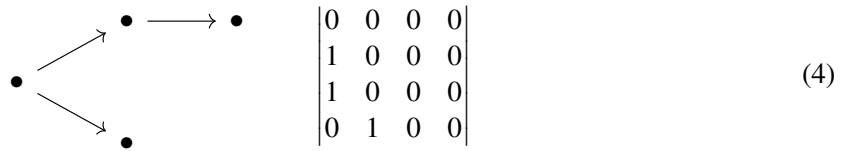
\begin{equation}
	\label{example:adjacency-matrix}
	\begin{tikzcd}[sep=small]
		&& {\bullet} && {\bullet} \\
		{\bullet} \\
		&& {\bullet}
		\arrow[from=2-1, to=1-3]
		\arrow[from=2-1, to=3-3]
		\arrow[from=1-3, to=1-5]
	\end{tikzcd}
	\qquad
	{
		\left|\begin{matrix}
			0 & 0 & 0 & 0 \\
			1 & 0 & 0 & 0 \\
			1 & 0 & 0 & 0 \\
			0 & 1 & 0 & 0 \\
		\end{matrix}\right|
	}
\end{equation}

Note that in this particular representation, there can be exactly one edge
between two nodes of the graph.
We can also introduce labeled edges by varying the semiring of $\Mat$: for
example, by considering matrices $\Mat_\Nat(K, K)$ we can consider edges to have
labels in the set $\{ 1, 2, \ldots \}$, with $0$ denoting no edge.
Consider for example the same graph as \eqref{example:adjacency-matrix} but with labeled edges:

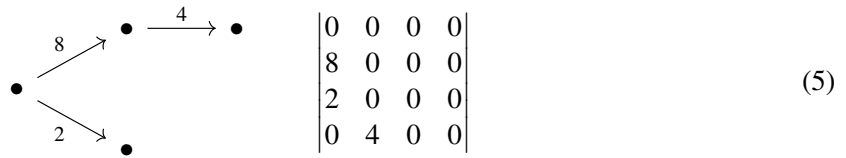
\begin{equation}
	\begin{tikzcd}[sep=small]
		&& {\bullet} && {\bullet} \\
		{\bullet} \\
		&& {\bullet}
		\arrow["8", from=2-1, to=1-3]
		\arrow["2"', from=2-1, to=3-3]
		\arrow["4", from=1-3, to=1-5]
	\end{tikzcd}
	\qquad
	{
		\left|\begin{matrix}
			0 & 0 & 0 & 0 \\
			8 & 0 & 0 & 0 \\
			2 & 0 & 0 & 0 \\
			0 & 4 & 0 & 0 \\
		\end{matrix}\right|
	}
\end{equation}

\section{Hypergraph Adjacency Representation}
\label{section:representation}
In this section we provide the main technical definition of the paper: the notion of Hypergraph Adjacency Representation (HAR). We begin by providing a roadmap to the formal definition. 
In a nutshell, the main hurdle is to adapt the approach to undirected hypergraphs in~\cite{MESH} (reported in Section~\ref{section:parahyprpoc}) to (directed) hypergraphs with interfaces. This will provide us a means of representing hypergraphs with interfaces as bipartite graphs, and thus as the corresponding adjacency matrices. As string diagrams can be identified as a certain class of hypergraphs with interfaces, this methodology will yield an implementation of string diagrams as adjacency matrices.
  
Before delving in the formal definition, the approach is best illustrated via an
example.
Recall the string diagram in~\eqref{equation:string-diagram-vs-hypergraph} (below left) with
its interpretation as a hypergraph with interfaces (below center).
Its bipartite graph encoding is displayed below right.
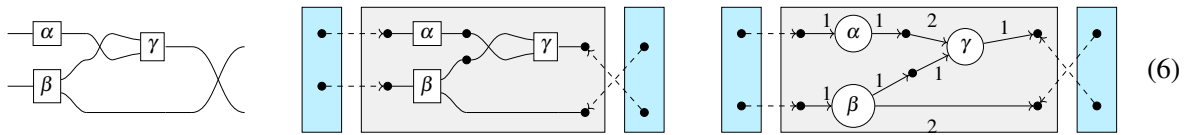
\begin{equation}
  \label{equation:example-encoded-hypergraph}
  \scalebox{0.7}{
    \begin{tikzpicture}
	\begin{pgfonlayer}{nodelayer}
		\node [style=none] (0) at (-4.5, 1.5) {};
		\node [style=morphism] (2) at (-3, 1.5) {$\alpha$};
		\node [style=none] (3) at (-1.5, 1.5) {};
		\node [style=none] (4) at (-4.5, -0.5) {};
		\node [style=morphism] (5) at (-3, -0.5) {$\beta$};
		\node [style=none] (6) at (-2.5, -0.25) {};
		\node [style=none] (7) at (-2.5, -0.75) {};
		\node [style=none] (8) at (-1.5, 0.5) {};
		\node [style=none] (9) at (-1.5, -1.5) {};
		\node [style=none] (10) at (-0.5, 0.5) {};
		\node [style=none] (11) at (-0.5, 1.5) {};
		\node [style=morphism] (12) at (1, 1) {$\gamma$};
		\node [style=none] (14) at (1.5, 1) {};
		\node [style=none] (15) at (0.75, 1.25) {};
		\node [style=none] (16) at (0.75, 0.75) {};
		\node [style=none] (17) at (1.5, -1.5) {};
		\node [style=none] (18) at (2.5, -1.5) {};
		\node [style=none] (19) at (2.5, 1) {};
		\node [style=none] (20) at (4.5, 1) {};
		\node [style=none] (21) at (4.5, -1.5) {};
	\end{pgfonlayer}
	\begin{pgfonlayer}{edgelayer}
		\draw (4.center) to (5);
		\draw (0.center) to (2);
		\draw (2) to (3.center);
		\draw [in=180, out=0, looseness=0.75] (7.center) to (9.center);
		\draw [in=180, out=0, looseness=0.75] (6.center) to (8.center);
		\draw [in=-180, out=0, looseness=0.75] (3.center) to (10.center);
		\draw [in=-180, out=0] (8.center) to (11.center);
		\draw [in=-180, out=-15] (11.center) to (15.center);
		\draw [in=180, out=15] (10.center) to (16.center);
		\draw (9.center) to (17.center);
		\draw (14.center) to (19.center);
		\draw (17.center) to (18.center);
		\draw [in=180, out=0, looseness=0.75] (18.center) to (20.center);
		\draw [in=-180, out=0, looseness=0.75] (19.center) to (21.center);
	\end{pgfonlayer}
\end{tikzpicture}
    \qquad
    \begin{tikzpicture}
	\begin{pgfonlayer}{nodelayer}
		\node [style=none] (24) at (-7.75, 2.5) {};
		\node [style=none] (25) at (-6.25, 2.5) {};
		\node [style=none] (26) at (-6.25, -2.25) {};
		\node [style=none] (27) at (-7.75, -2.25) {};
		\node [style=none] (28) at (4.5, 2.5) {};
		\node [style=none] (29) at (6, 2.5) {};
		\node [style=none] (30) at (6, -2.25) {};
		\node [style=none] (31) at (4.5, -2.25) {};
		\node [style=none] (32) at (-5.5, 2.5) {};
		\node [style=none] (33) at (-5.5, -2.25) {};
		\node [style=none] (34) at (3.75, -2.25) {};
		\node [style=none] (35) at (3.75, 2.5) {};
		\node [style=node] (0) at (-4.5, 1.5) {};
		\node [style=morphism] (2) at (-3, 1.5) {$\alpha$};
		\node [style=node] (3) at (-1.5, 1.5) {};
		\node [style=node] (4) at (-4.5, -0.5) {};
		\node [style=morphism] (5) at (-3, -0.5) {$\beta$};
		\node [style=none] (6) at (-2.5, -0.25) {};
		\node [style=none] (7) at (-2.5, -0.75) {};
		\node [style=node] (8) at (-1.5, 0.5) {};
		\node [style=none] (9) at (-1.5, -1.5) {};
		\node [style=none] (10) at (0, 0.5) {};
		\node [style=none] (11) at (0, 1.5) {};
		\node [style=morphism] (12) at (1.5, 1) {$\gamma$};
		\node [style=none] (14) at (2, 1) {};
		\node [style=none] (15) at (1.25, 1.25) {};
		\node [style=none] (16) at (1.25, 0.75) {};
		\node [style=none] (17) at (2, -1.5) {};
		\node [style=node] (18) at (3, -1.5) {};
		\node [style=node] (19) at (3, 1) {};
		\node [style=node] (20) at (5.25, 1) {};
		\node [style=node] (21) at (5.25, -1.5) {};
		\node [style=node] (22) at (-7, 1.5) {};
		\node [style=node] (23) at (-7, -0.5) {};
	\end{pgfonlayer}
	\begin{pgfonlayer}{edgelayer}
		\draw [style=bluefill] (25.center)
			 to (26.center)
			 to (27.center)
			 to (24.center)
			 to cycle;
		\draw [style=bluefill] (29.center)
			 to (30.center)
			 to (31.center)
			 to (28.center)
			 to cycle;
		\draw [style=greyfill] (33.center)
			 to (32.center)
			 to (35.center)
			 to (34.center)
			 to cycle;
		\draw (4) to (5);
		\draw (0) to (2);
		\draw (2) to (3);
		\draw [in=180, out=0, looseness=0.75] (7.center) to (9.center);
		\draw [in=180, out=0, looseness=0.75] (6.center) to (8);
		\draw [in=-180, out=0, looseness=0.75] (3) to (10.center);
		\draw [in=-180, out=0] (8) to (11.center);
		\draw [in=-180, out=-15] (11.center) to (15.center);
		\draw [in=180, out=15] (10.center) to (16.center);
		\draw (9.center) to (17.center);
		\draw (14.center) to (19);
		\draw (17.center) to (18);
		\draw [style=dashpoint] (21) to (19);
		\draw [style=dashpoint] (20) to (18);
		\draw [style=dashpoint] (23) to (4);
		\draw [style=dashpoint] (22) to (0);
	\end{pgfonlayer}
\end{tikzpicture}
    \qquad
    \begin{tikzpicture}
	\begin{pgfonlayer}{nodelayer}
		\node [style=none] (67) at (-5.75, 2.5) {};
		\node [style=none] (68) at (-5.75, -2.25) {};
		\node [style=none] (69) at (4.75, -2.25) {};
		\node [style=none] (70) at (4.75, 2.5) {};
		\node [style=none] (63) at (-8, 2.5) {};
		\node [style=none] (64) at (-6.5, 2.5) {};
		\node [style=none] (65) at (-6.5, -2.25) {};
		\node [style=none] (66) at (-8, -2.25) {};
		\node [style=none] (71) at (5.5, 2.5) {};
		\node [style=none] (72) at (7, 2.5) {};
		\node [style=none] (73) at (7, -2.25) {};
		\node [style=none] (74) at (5.5, -2.25) {};
		\node [style=node] (28) at (-5, 1.5) {};
		\node [style=edge] (29) at (-3, 1.5) {$\alpha$};
		\node [style=node] (30) at (-1, 1.5) {};
		\node [style=node] (31) at (-5, -1.25) {};
		\node [style=edge] (32) at (-3, -1.25) {$\beta$};
		\node [style=edge] (39) at (1.25, 1) {$\gamma$};
		\node [style=node] (44) at (4, -1.25) {};
		\node [style=node] (45) at (4, 1.5) {};
		\node [style=node] (46) at (6.25, 1.5) {};
		\node [style=node] (47) at (6.25, -1.25) {};
		\node [style=node] (48) at (-7.25, 1.5) {};
		\node [style=node] (49) at (-7.25, -1.25) {};
		\node [style=node] (54) at (-0.75, 0) {};
		\node [style=none] (55) at (-4, 2) {1};
		\node [style=none] (56) at (-4, -0.75) {1};
		\node [style=none] (57) at (-2, -0.25) {1};
		\node [style=none] (58) at (0, -2) {2};
		\node [style=none] (59) at (0, 2) {2};
		\node [style=none] (60) at (-2, 2) {1};
		\node [style=none] (61) at (0.25, 0) {1};
		\node [style=none] (62) at (2.75, 1.75) {1};
	\end{pgfonlayer}
	\begin{pgfonlayer}{edgelayer}
		\draw [style=greyfill] (68.center)
			 to (67.center)
			 to (70.center)
			 to (69.center)
			 to cycle;
		\draw [style=bluefill] (64.center)
			 to (65.center)
			 to (66.center)
			 to (63.center)
			 to cycle;
		\draw [style=bluefill] (73.center)
			 to (74.center)
			 to (71.center)
			 to (72.center)
			 to cycle;
		\draw [style=pointy] (31) to (32);
		\draw [style=pointy] (28) to (29);
		\draw [style=pointy] (29) to (30);
		\draw [style=dashpoint] (47) to (45);
		\draw [style=dashpoint] (46) to (44);
		\draw [style=dashpoint] (49) to (31);
		\draw [style=dashpoint] (48) to (28);
		\draw [style=pointy] (32) to (54);
		\draw [style=pointy] (32) to (44);
		\draw [style=pointy] (54) to (39);
		\draw [style=pointy] (30) to (39);
		\draw [style=pointy] (39) to (45);
	\end{pgfonlayer}
\end{tikzpicture}
  }
\end{equation}
Note this is similar to the bipartite graph encoding shown
in (\ref{equation:bipartite-encoding-undirected-hypergraph}), as made evident when we rearrange the bipartite graph of~\eqref{equation:example-encoded-hypergraph} as follows:
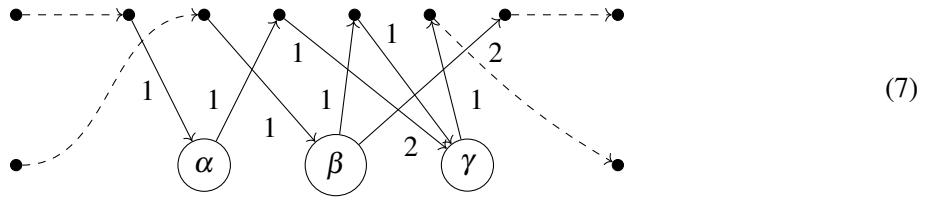
\begin{equation}\label{eq:bipartite-stringdiagram}
  \begin{tikzpicture}
	\begin{pgfonlayer}{nodelayer}
		\node [style=node] (0) at (-5, 2) {};
		\node [style=node] (1) at (-3, 2) {};
		\node [style=node] (2) at (-1, 2) {};
		\node [style=node] (3) at (1, 2) {};
		\node [style=node] (4) at (3, 2) {};
		\node [style=node] (5) at (5, 2) {};
		\node [style=edge] (6) at (-3, -2) {$\alpha$};
		\node [style=edge] (7) at (0.5, -2) {$\beta$};
		\node [style=edge] (8) at (4, -2) {$\gamma$};
		\node [style=none] (9) at (-4.5, 0) {1};
		\node [style=none] (10) at (-2.75, -0.25) {1};
		\node [style=none] (11) at (-1.25, -1) {1};
		\node [style=none] (12) at (2.5, -1.5) {2};
		\node [style=none] (13) at (4.75, 1) {2};
		\node [style=none] (14) at (-0.5, 1) {1};
		\node [style=none] (15) at (0.25, -0.25) {1};
		\node [style=none] (16) at (2, 1.5) {1};
		\node [style=none] (17) at (4.25, -0.25) {1};
		\node [style=node] (19) at (8, 2) {};
		\node [style=node] (20) at (-8, 2) {};
		\node [style=node] (21) at (-8, -2) {};
		\node [style=node] (22) at (8, -2) {};
	\end{pgfonlayer}
	\begin{pgfonlayer}{edgelayer}
		\draw [style=pointy] (0) to (6);
		\draw [style=pointy] (1) to (7);
		\draw [style=pointy] (6) to (2);
		\draw [style=pointy] (7) to (3);
		\draw [style=pointy] (7) to (5);
		\draw [style=pointy] (2) to (8);
		\draw [style=pointy] (3) to (8);
		\draw [style=pointy] (8) to (4);
		\draw [style=dashpoint] (5) to (19);
		\draw [style=dashpoint] (20) to (0);
		\draw [style=dashpoint, in=-180, out=0] (21) to (1);
		\draw [style=dashpoint, in=150, out=-45] (4) to (22);
	\end{pgfonlayer}
\end{tikzpicture}
\end{equation}
The differences are (i) the presence of interfaces (needed because we are interested in \emph{composing} these structures), (ii) the labeling of $\hyperedge$-vertices with $\Sigma$-operations, and (iii) the labeling of edges with natural numbers. The latter information indicates the position, in the original hypergraph with interfaces, of a node in the source/target lists of a hyperedge. For instance, the edge labeled with $2$ indicates that, in the original hypergraph, the target node of hyperedge $\alpha$ is in the second position in the source list of hyperedge~$\gamma$.

The next step is translating this bipartite graph into an adjacency matrix (along the lines of Section~\ref{section:adjacencymatrices}), together with information on what are the interfaces of the graph. This leads us to the data structure called HAR: a 4-tuple $(M, L, R, N)$,
with $M$ serving double-duty as the adjacency matrix and edge-label data, $N$ a
vector of node labels, and $L$ and $R$ permutation matrices reordering $M$ so
that left boundary nodes are the first and right boundary nodes the last, respectively.
Returning to our example hypergraph in \eqref{equation:example-encoded-hypergraph},
we represent it with the following data (in which we write $N$ twice for clarity):

\begin{align*}
  M & = {
    \left|\begin{matrix}
      0 & 0 & 0 & 0 & 0 & 0 & 0 & 0 & 0 \\
      0 & 0 & 0 & 0 & 0 & 0 & 0 & 0 & 0 \\
      1 & 0 & 0 & 0 & 0 & 0 & 0 & 0 & 0 \\
      0 & 1 & 0 & 0 & 0 & 0 & 0 & 0 & 0 \\
      0 & 0 & 1 & 0 & 0 & 0 & 0 & 0 & 0 \\
      0 & 0 & 0 & 1 & 0 & 0 & 0 & 0 & 0 \\
      0 & 0 & 0 & 0 & 2 & 1 & 0 & 0 & 0 \\
      0 & 0 & 0 & 0 & 0 & 0 & 1 & 0 & 0 \\
      0 & 0 & 0 & 2 & 0 & 0 & 0 & 0 & 0 \\
    \end{matrix}\right|
    \qquad
    N =
    \left|\begin{matrix}
      \hypernode \\ \hypernode \\ {\alpha} \\
      {\beta} \\ \hypernode \\ \hypernode    \\
      {\gamma} \\ \hypernode \\ \hypernode
    \end{matrix}\right|
    \qquad
    L = \id_9
    \qquad
    R = {
      \left|\begin{matrix}
        \id_7  & \Mzero \\
        \Mzero & \twist_{1,1} \\
      \end{matrix}\right|
    }
  } \\
  N & = {
    \left|\begin{matrix}
      \hypernode  & \hypernode & \alpha &
      \beta     & \hypernode & \hypernode        &
      \gamma    & \hypernode & \hypernode
    \end{matrix}\right|
  }
\end{align*}

We can read the columns of $M$ as the outgoing edges for a particular node.
See for example the column for $\beta$, which has two outgoing edges
labeled $1$ and $2$, both of which connect to nodes labeled $\hypernode$.
Note that $L$ is the identity matrix: this means that the left interface nodes
appear first, and moreover they appear in the same order as in the interface.
Hence, the first $2$ rows contain only zeros because the left interface nodes
have no incoming edges.
On the other hand $R$ is the block matrix
${\scriptsize
\left|\begin{matrix}
  \id    & \Mzero \\
  \Mzero & \twist \\
\end{matrix}\right|
}$
and so while the final two nodes are the right interface nodes, their order in
the interface is swapped.

\subsection{Main Definition}

We can now give our main definition---for background on matrix notation see
Section \ref{section:props-of-matrices}.

\begin{definition}[Hypergraph Adjacency Representation]
  \label{definition:har}
  Fix a monoidal signature $\Sigma$.
  A hypergraph adjacency representation of type $\ObA \to \ObB$ is written $\Har_{\ObA, \ObB}$
  and consists of the following data:

  \begin{itemize}[label=-]
    \item \textbf{Size} $K \in \Nat$
    \item \textbf{Labeled Adjacency Matrix} $M \in \Mat_{\Nat}(K, K)$
    \item \textbf{Left Permutation} $L \in \Mat_{\Bool}(K, K)$
    \item \textbf{Right Permutation} $R \in \Mat_{\Bool}(K, K)$
    \item \textbf{Node Labels} $N \in (\{ \hypernode \} + (\{ \hyperedge \} \times \Sigma))^K$
  \end{itemize}

  \noindent satisfying the following conditions:

  \begin{itemize}[label=-]
    \item The graph represented by $M$ is acyclic.
    \item The matrix $L^T \cp M \cp L$ is ordered such that
          the \textbf{first} $m$ nodes are the \deftext{left interface nodes}
    \item The matrix $R^T \cp M \cp R$ is ordered such that
      the \textbf{last} $n$ nodes are the \deftext{right interface nodes}.
    \item If a node labeled $\hypernode$ is not an interface node, then it has
      exactly one incoming and outgoing edge.
    \item For each vertex $v$ labeled $(\hyperedge, g)$ with $g$ having arity/coarity $m, n$,
      \begin{itemize}[label=-]
        \item $v$ has incoming edges $e_1 \ldots e_m$ with labels $1 \ldots m$ respectively.
        \item $v$ has outgoing edges $e_1 \ldots e_n$ with labels $1 \ldots n$ respectively.
      \end{itemize}
  \end{itemize}

\end{definition}

\subsection{Permutation Equivalence and Boundary Orderings}

In Section \ref{section:operations} we will see that composition of \Har{}s is only
associative up to isomorphism.
Therefore, in order to form a category of \Har{}s, we will quotient by the
following equivalence relation,\footnote{
  We could take the alternative perspective that \Har{} forms a weak 2-category
  with permutation matrices as 2-cells, but we will take the equivalence relation
  perspective to simplify our presentation.
} which equates \Har{}s having isomorphic graphs.

\begin{definition}[Permutation Equivalence]
  \label{definition:permutation-equivalence}
  $f, g : \Har_{\ObA,\ObB}$ are \deftext{equivalent up to permutation $P$},
  denoted $f \permeq{P} g$,
  when $P$ is a permutation matrix such that the following conditions hold:
  \[
    g_M = P^T \cp f_M \cp P \qquad\qquad g_L = P^T \cp f_L \qquad\qquad g_R = P^T \cp f_R \qquad\qquad g_N = f_N \cp P
  \]
\end{definition}

\begin{remark}
  Note that this definition ensures that if $f \permeq{P} g$
  then $g_M$ is a graph isomorphic to $f_M$
  and also that the interfaces of $f$ and $g$ are the same.
\end{remark}

\begin{proposition}[Permutation Equivalence is an Equivalence Relation]
  Fix some $f, g\in \Har_{\ObA, \ObB}$.
  Then there exists an equivalence relation denoted $\sim$
  such that $f \sim g$ iff there exists
  some permutation matrix $P \in \Mat_\Bool(K, K)$ such that
  $f \permeq{P} g$ (cf. Definition \ref{definition:permutation-equivalence})
\end{proposition}

\begin{proof}
  Clearly $\permeq{}$ is reflexive because $f \permeq{\id} f$.
  Further, it is symmetric because if $f \permeq{P} g$, then $g \permeq{P^T} f$.
  Finally, transitivity follows from matrix composition:
  If $f \permeq{P} g$ and $g \permeq{Q} h$, then $f \permeq{P \cp Q} h$.
\end{proof}

This definition means that each $f : \Har_{\ObA, \ObB}$ can be put into an equivalent
left (resp.\ right) boundary order by permuting by $f_L$ (resp.\ $f_R$).
We will make heavy use of these particular permutations in defining composition
and tensor product, so we define them explicitly.

\begin{definition}
  \label{definition:left-boundary-order}
  The \deftext{left boundary order} of $f \in \Har_{\ObA, \ObB}$ is denoted $\lbo{f}$ and has the
  following data:

  \[
    \lbo{f}_M = f_L^T \cp f_M \cp f_L \qquad
    \lbo{f}_L = \id_{f_K} \qquad
    \lbo{f}_R = f_L^T \cp f_R \qquad
    \lbo{f}_N = f_N \cp f_L
  \]
\end{definition}

\begin{definition}
  \label{definition:right-boundary-order}
  The \deftext{right boundary order} of $f \in \Har_{\ObA, \ObB}$ is denoted $\rbo{f}$ and has the
  following data:
  \[
    \rbo{f}_M = f_R^T \cp f_M \cp f_R \qquad
    \rbo{f}_L = f_R^T \cp f_L \qquad
    \rbo{f}_R = \id_{f_K} \qquad
    \rbo{f}_N = f_N \cp f_R
  \]
\end{definition}

\begin{remark}
  Note that by definition $L(f) \permeq{f_L} f$ and vice-versa, $R(f) \permeq{f_R} f$.
\end{remark}

\section{Operations on \Har{}s}
\label{section:operations}
The main motivation for introducing $\Har$s is providing an efficient implementation for composing string diagrams. To this aim, in this section we define the operations for constructing and combining $\Har$s. These developments will also allow us to prove that $\Har$s form a category, in the next section.



\smallskip
\begin{definition}~
  \label{definition:operations}
\begin{itemize}
	\item The identity $\Har$ of type $\ObA \to \ObA$ is $(\Mzero_{\ObA, \ObA}, \id_\ObA, \id_\ObA, \Mzero_{\ObA, 1})$.
	\item   The symmetry $\twist_{\ObA,\ObB} : \ObA \otimes \ObB \to \ObB \otimes \ObA$
  is $(\Mzero_{\ObA+\ObB,\ObA+\ObB}, \id_{\ObA+\ObB}, P, \Mzero_{\ObA+\ObB,1})$
  with $P$ the block matrix
  $
   {\scriptsize \left|\begin{matrix}
      \Mzero    & \id_\ObB \\
      \id_\ObA  & \Mzero
    \end{matrix}\right|}
  $.
	\item Given an operation $g \colon \ObA \to \ObB \in \Sigma$,
  the `singleton'\footnote{
    The name `singleton' refers to the fact that such a \Har{} contains a single generator.
    We choose this name based on its common usage in Haskell libraries for a
    function creating a datastructure (e.g. a set) with a single element.
  }\Har{} is given by
  $(M, \id_K, \id_K, N)$,
  with size
  $K = \ObA + \ObB + 1$,
  node labels
  $N = (\Mzero_{1, \ObA}, g, \Mzero_{1, \ObB})$
  and $M$ the block matrix
  $
    {\scriptsize\left|\begin{matrix}
      0 & 0 & 0 \\
      S & 0 & 0 \\
      0 & T & 0
    \end{matrix}\right|}
  $,
  where $S \in \Mat_\Nat(1, \ObA)$ is the row vector
  $(1, 2, \ldots, \ObA)$
  and $T \in \Mat_\Nat(\ObB, 1)$ the column vector
  $(1, 2, \ldots, \ObB)$.
  \item Let $f : \ObA_1 \to \ObB_1$ and $g : \ObA_2 \to \ObB_2$ be \Har{}s.
The tensor product $f \otimes g$ is component-wise as follows.

\[
  (f \otimes g)_M = \scalebox{0.75}{\begin{tikzpicture}
	\begin{pgfonlayer}{nodelayer}
		\node [style=morphism] (0) at (0, 1) {$f_M$};
		\node [style=morphism] (1) at (0, -1) {$g_M$};
		\node [style=none] (2) at (-2, 1) {};
		\node [style=none] (3) at (-2, -1) {};
		\node [style=none] (4) at (2, -1) {};
		\node [style=none] (5) at (2, 1) {};
		\node [style=none] (6) at (-3, 1) {$f_K$};
		\node [style=none] (7) at (3, 1) {$f_K$};
		\node [style=none] (8) at (-3, -1) {$g_K$};
		\node [style=none] (9) at (3, -1) {$g_K$};
	\end{pgfonlayer}
	\begin{pgfonlayer}{edgelayer}
		\draw (2.center) to (0);
		\draw (0) to (5.center);
		\draw (3.center) to (1);
		\draw (1) to (4.center);
	\end{pgfonlayer}
\end{tikzpicture}} \qquad
  (f \otimes g)_L = \scalebox{0.75}{\begin{tikzpicture}
	\begin{pgfonlayer}{nodelayer}
		\node [style=morphism] (0) at (-0.75, 1) {$f_L$};
		\node [style=morphism] (1) at (-0.75, -1) {$g_L$};
		\node [style=none] (2) at (-2.25, 1) {};
		\node [style=none] (3) at (-2.25, -1) {};
		\node [style=none] (5) at (0.5, 1.25) {};
		\node [style=none] (6) at (-3, 1) {$f_K$};
		\node [style=none] (7) at (3.25, 1.25) {$\ObA_1$};
		\node [style=none] (8) at (-3, -1) {$g_K$};
		\node [style=none] (9) at (3.25, 0.5) {$\ObA_2$};
		\node [style=none] (10) at (-0.25, 1.25) {};
		\node [style=none] (11) at (-0.25, 0.75) {};
		\node [style=none] (12) at (0.5, 0.5) {};
		\node [style=none] (13) at (0.5, -0.5) {};
		\node [style=none] (14) at (-0.25, -0.75) {};
		\node [style=none] (15) at (-0.25, -1.25) {};
		\node [style=none] (16) at (0.5, -1.25) {};
		\node [style=none] (17) at (2.5, 0.5) {};
		\node [style=none] (18) at (2.5, -0.5) {};
		\node [style=none] (19) at (2.5, -1.25) {};
		\node [style=none] (20) at (2.5, 1.25) {};
		\node [style=none] (21) at (3.75, -0.5) {$f_K - \ObA_1$};
		\node [style=none] (23) at (3.75, -1.25) {$g_K - \ObA_2$};
	\end{pgfonlayer}
	\begin{pgfonlayer}{edgelayer}
		\draw (2.center) to (0);
		\draw (3.center) to (1);
		\draw (10.center) to (5.center);
		\draw [in=-180, out=0] (11.center) to (12.center);
		\draw [in=180, out=0] (14.center) to (13.center);
		\draw (15.center) to (16.center);
		\draw [in=180, out=0] (12.center) to (18.center);
		\draw [in=180, out=0] (13.center) to (17.center);
		\draw (16.center) to (19.center);
		\draw (5.center) to (20.center);
	\end{pgfonlayer}
\end{tikzpicture}} \qquad
  (f \otimes g)_R = \scalebox{0.75}{\begin{tikzpicture}
	\begin{pgfonlayer}{nodelayer}
		\node [style=morphism] (0) at (-0.75, 1) {$f_R$};
		\node [style=morphism] (1) at (-0.75, -1) {$g_R$};
		\node [style=none] (2) at (-2.25, 1) {};
		\node [style=none] (3) at (-2.25, -1) {};
		\node [style=none] (5) at (0.5, 1.25) {};
		\node [style=none] (6) at (-3, 1) {$f_K$};
		\node [style=none] (7) at (3.25, -0.5) {$\ObB_1$};
		\node [style=none] (8) at (-3, -1) {$g_K$};
		\node [style=none] (9) at (3.25, -1.25) {$\ObB_2$};
		\node [style=none] (10) at (-0.25, 1.25) {};
		\node [style=none] (11) at (-0.25, 0.75) {};
		\node [style=none] (12) at (0.5, 0.5) {};
		\node [style=none] (13) at (0.5, -0.5) {};
		\node [style=none] (14) at (-0.25, -0.75) {};
		\node [style=none] (15) at (-0.25, -1.25) {};
		\node [style=none] (16) at (0.5, -1.25) {};
		\node [style=none] (17) at (2.5, 0.5) {};
		\node [style=none] (18) at (2.5, -0.5) {};
		\node [style=none] (19) at (2.5, -1.25) {};
		\node [style=none] (20) at (2.5, 1.25) {};
		\node [style=none] (21) at (3.75, 1.25) {$f_K - \ObB_1$};
		\node [style=none] (23) at (3.75, 0.5) {$g_K - \ObB_2$};
	\end{pgfonlayer}
	\begin{pgfonlayer}{edgelayer}
		\draw (2.center) to (0);
		\draw (3.center) to (1);
		\draw (10.center) to (5.center);
		\draw [in=-180, out=0] (11.center) to (12.center);
		\draw [in=180, out=0] (14.center) to (13.center);
		\draw (15.center) to (16.center);
		\draw [in=180, out=0] (12.center) to (18.center);
		\draw [in=180, out=0] (13.center) to (17.center);
		\draw (16.center) to (19.center);
		\draw (5.center) to (20.center);
	\end{pgfonlayer}
\end{tikzpicture}} \qquad
\]
with $(f \otimes g)_N$ given by appending $f_N$ and $g_N$, i.e., the block vector $( f_N \quad g_N )$.
\item Let $f : \ObA \to \ObB$ and $g : \ObB \to \ObC$ be \Har{}s.
Composition is defined component-wise as follows:

\[
  (f \cp g)_M = \scalebox{0.65}{\begin{tikzpicture}
	\begin{pgfonlayer}{nodelayer}
		\node [style=morphism] (0) at (0, 1) {$R(f)_M$};
		\node [style=none] (2) at (-3, 1.25) {};
		\node [style=none] (5) at (3, 1) {};
		\node [style=none] (6) at (-4.25, 1.25) {$f_K - \ObB$};
		\node [style=none] (7) at (3.75, 1) {$f_K$};
		\node [style=none] (10) at (-0.5, 1.25) {};
		\node [style=none] (11) at (-0.5, 0.75) {};
		\node [style=node] (12) at (-2, 0.75) {};
		\node [style=morphism] (13) at (0, -1) {$L(g)_M$};
		\node [style=none] (14) at (3, -1.25) {};
		\node [style=none] (15) at (-3, -1) {};
		\node [style=none] (16) at (4.25, -1.25) {$g_K - \ObB$};
		\node [style=none] (17) at (-3.75, -1) {$g_K$};
		\node [style=none] (18) at (0.5, -1.25) {};
		\node [style=none] (19) at (0.5, -0.75) {};
		\node [style=node] (20) at (2, -0.75) {};
	\end{pgfonlayer}
	\begin{pgfonlayer}{edgelayer}
		\draw (0) to (5.center);
		\draw (12) to (11.center);
		\draw (2.center) to (10.center);
		\draw (13) to (15.center);
		\draw (20) to (19.center);
		\draw (14.center) to (18.center);
	\end{pgfonlayer}
\end{tikzpicture}} \qquad
  (f \cp g)_L = \scalebox{0.65}{\begin{tikzpicture}
	\begin{pgfonlayer}{nodelayer}
		\node [style=morphism] (0) at (0, 0.75) {$R(f)_L$};
		\node [style=none] (2) at (-2, 0.75) {};
		\node [style=none] (3) at (-2, -0.75) {};
		\node [style=none] (4) at (2, -0.75) {};
		\node [style=none] (5) at (2, 0.75) {};
		\node [style=none] (6) at (-2.75, 0.75) {$f_K$};
		\node [style=none] (7) at (2.75, 0.75) {$f_K$};
		\node [style=none] (8) at (-3.25, -0.75) {$g_K - \ObB$};
		\node [style=none] (10) at (3.25, -0.75) {$g_K - \ObB$};
	\end{pgfonlayer}
	\begin{pgfonlayer}{edgelayer}
		\draw (2.center) to (0);
		\draw (0) to (5.center);
		\draw (3.center) to (4.center);
	\end{pgfonlayer}
\end{tikzpicture}} \qquad
  (f \cp g)_R = \scalebox{0.65}{\begin{tikzpicture}
	\begin{pgfonlayer}{nodelayer}
		\node [style=morphism] (1) at (0, -0.75) {$L(g)_R$};
		\node [style=none] (2) at (-2, 0.75) {};
		\node [style=none] (3) at (-2, -0.75) {};
		\node [style=none] (4) at (2, -0.75) {};
		\node [style=none] (5) at (2, 0.75) {};
		\node [style=none] (6) at (-3.25, 0.75) {$f_K - \ObB$};
		\node [style=none] (7) at (3.25, 0.75) {$f_K - \ObB$};
		\node [style=none] (8) at (-2.75, -0.75) {$g_K$};
		\node [style=none] (9) at (2.75, -0.75) {$g_K$};
	\end{pgfonlayer}
	\begin{pgfonlayer}{edgelayer}
		\draw (3.center) to (1);
		\draw (1) to (4.center);
		\draw (2.center) to (5.center);
	\end{pgfonlayer}
\end{tikzpicture}} \qquad
\]

with $(f \cp g)_N$ given by appending $\rbo{f}_N$ and $\lbo{g}_N(b:)$, where $x(b:)$ denotes all
but the first $b$ elements of the array $x$.
Alternatively, one may regard $(f \cp g)_N$ as a diagonal matrix, and define
composition as for $(f \cp g)_M$.
\end{itemize}
\end{definition}

\begin{remark}
  Note that in the definition of $(f \cp g)_M$, the morphisms
  $\scalebox{0.5}{\begin{tikzpicture}
	\begin{pgfonlayer}{nodelayer}
		\node [style=none] (0) at (-1, -0.5) {};
		\node [style=none] (1) at (1, -0.5) {};
		\node [style=none] (2) at (-1, 0.5) {};
		\node [style=node] (3) at (0.5, 0.5) {};
	\end{pgfonlayer}
	\begin{pgfonlayer}{edgelayer}
		\draw (0.center) to (1.center);
		\draw (2.center) to (3);
	\end{pgfonlayer}
\end{tikzpicture}
}$ and
  $\scalebox{0.5}{\begin{tikzpicture}
	\begin{pgfonlayer}{nodelayer}
		\node [style=none] (0) at (1, 0.5) {};
		\node [style=none] (1) at (-1, 0.5) {};
		\node [style=none] (2) at (1, -0.5) {};
		\node [style=node] (3) at (-0.5, -0.5) {};
	\end{pgfonlayer}
	\begin{pgfonlayer}{edgelayer}
		\draw (0.center) to (1.center);
		\draw (2.center) to (3);
	\end{pgfonlayer}
\end{tikzpicture}
}$ represent the
  projection $\pi_2 : A \times B \to B$ and embedding $\iota_1 : A \to A \times B$ morphisms, respectively.
  One may think of the composition $M \cp \pi_1$ as selecting the last columns
  of $M$ and $\iota_0 \cp M$ as selecting the first rows.
\end{remark}

\section{Adequacy of the \Har{}-Implementation}
\label{section:category}
In this section we show how the interpretation of string diagrams as $\Har$s can be described as a \emph{full and faithful} functor between PROPs, meaning that our implementation is actually a 1-to-1 correspondence. As a preliminary step, we need to show how $\Har{}$s form a category.
\begin{proposition} 
  There is a PROP $\Har{}_\Sigma$ whose morphisms $\ObA \to \ObB$ are
  equivalence classes of values $\Har_{\ObA, \ObB}$ under the equivalence relation $\permeq{}$, and identity, symmetries, composition and tensor product are as defined in Section~\ref{section:operations}.
\end{proposition}

\begin{proof}
  We give a graphical proof that composition is assocative up to permutation in
  the full version of our paper \cite[Appendix A.6]{coc}.

  It is straightforward to check that $f \cp \id = f$, and similarly one
  can check that $\twist \cp \twist \permeq{\twist} \id$.
  Finally, one can see that the tensor product is associative essentially
  because the direct sum is.
\end{proof}

\begin{definition} Let $\ToHar{\cdot} : \Syn{\Sigma} \to \Har_\Sigma$ be the identity-on-objects symmetric monoidal functor freely obtained by the mapping of operations $g \in \Sigma$ to singleton $\Har$s, as defined in Section~\ref{section:operations}. 
\end{definition}

\begin{proposition}\label{prop:iso} $\ToHar{\cdot} : \Syn{\Sigma} \to \Har_\Sigma$ is an isomorphism of PROPs.
\end{proposition}

We now give a sketch of our proof, leaving the full details to
the full version of our paper \cite[Appendix B]{coc}.

\begin{proof}
  Thanks to Proposition \cite[B.6]{coc},
  it suffices to show that:
  \begin{itemize}[label=-]
    \item There is a symmetric monoidal functor $\FromHar{\cdot} \colon 
       \Har_\Sigma \to \Syn{\Sigma}$,
    \item $\Har_\Sigma$ is generated by the singleton $\Har$s corresponding to
      the operations $g \in \Sigma$, and
    \item $\FromHar{\ToHar{g}} = g$ for $g \in \Sigma$.
  \end{itemize}
  Essentially, the idea is to show that $\Har_\Sigma$ is just a `relabeling of
  generators' of $\Syn{\Sigma}$.
\end{proof}

\section{Complexity}
\label{section:complexity}
We now give the time complexity of the composition and tensor product operations
defined in Section \ref{section:operations}.
We give empirical results to validate our claims in Section \ref{section:empirical}.

Naively, since our algorithm is expressed in terms of matrix multiplication,
it should have a time complexity of at best $O(n^{2.3728596})$ (at time of
writing~\cite{alman2020refined}).
However, we can do significantly better by exploiting the high degree of
\emph{sparsity} of the matrices of a \Har{}.

Concretely, observe that for a finite monoidal signature $\Sigma$ and $f \in \Har(\ObA, \ObB)$,
one can guarantee that the number of non-zero elements in $f_M$
is $O(f_K)$:

\begin{proposition}[Bounded sparsity]
  \label{proposition:bounded-sparsity}
  Fix a finite monoidal signature $\Sigma$ and let $f$ be a $\Har$.
  Now Let $m$ be the largest arity of any generator $g \in \Sigma$ and $n$ the
  largest \emph{coarity}.
  Then the rows of $f_M$ have at most $m$ non-zero elements,
  the columns at most $n$ non-zero elements,
  and $f_M$ has $O(f_K)$ non-zero elements.
\end{proposition}

\begin{proof}
  By Definition \ref{definition:har}, each vertex $v$ in the graph represented
  by $f_M$ must have exactly $m$ incoming and $n$ outgoing edges.
  These edges correspond to the non-zero rows and columns of $f_M$, respectively,
  and so the non-zero elements of each row (resp.\ column) is at most $m$ (resp.\ $n$).
\end{proof}

Now, it happens that the time complexity of the `naive' sparse matrix
multiplication algorithm~\cite{GustavsonSparse} is essentially linear in the
number of \emph{non-trivial multiplications} required---that is, those scalar
multiplications where neither multiplicand is zero.
From this fact and the property of bounded sparsity, it follows that both
composition and tensor product of \Har{}s are linear-time operations.
To make this clear, we introduce the following proposition:

\begin{proposition}[Permutation of $\Har{}$ has linear complexity]
  \label{proposition:permutation-linear}
  Choose some $f \in \Har_\Sigma$ and a permutation matrix $P \in \Mat(f_K, f_K)$.
  Then $P \cp f_M$ and $f_M \cp P$ can be computed in linear time.
\end{proposition}

\begin{proof}
  For matrices $A, B \in \Har(k, k)$, the complexity of Gustavson's sparse
  matrix multiplication routine~\cite{GustavsonSparse} is $O(2 k + \nnz(A) + m)$.
  Here $\nnz(A)$ is the number of non-zero entries of $A$
  and $m$ is the number of non-trivial multiplications required.

  By the bounded sparsity property
  (Proposition \ref{proposition:bounded-sparsity}),
  one can see that computing a row of the matrix $f_M \cp P$ requires only a
  \emph{constant} number of non-trivial multiplications, and further $\nnz(f_M)$ is
  $O(f_K)$.
  Thus, computing $f_M \cp P$ is $O(f_K)$.

  Alternatively, one may also see that linear complexity is possible using
  Gustavson's \Halfperm{} algorithm~\cite{GustavsonSparse},
  which can compute $P \cp f_M \cp Q^T$ in $O(\nnz(f_M))$ operations.
  Since $\nnz(f_M)$ is $O(f_K)$, this operation has linear complexity.
\end{proof}

Using Proposition \ref{proposition:permutation-linear} we can now show that
composition and tensor product have linear time complexity.

\begin{proposition}[Tensor Product of \Har{}s $f \otimes g$ is $O(f_K + g_K)$]
  \label{proposition:tensor-product-linear}
  Given $f \in \Har(\ObA_1, \ObB_1)$ and $g \in \Har(\ObA_2, \ObB_2)$,
  computation of $f \otimes g$ is $O(f_K + g_K)$.
\end{proposition}

\begin{proof}
  It is clear from definition \ref{definition:har} that each component of $f
  \otimes g$ is computed either as a direct sum or a multiplication of
  permutation matrices of size $f_K + g_K$.
  Since each of these operations is $O(f_K + g_K)$, it is clear that the whole
  operation is as well.
\end{proof}

\begin{proposition}[Composition of \Har{}s $f \cp g$ is $O(f_K + g_K)$]
  Given $f \in \Har(\ObA, \ObB)$ and $g \in \Har(\ObB, \ObC)$,
  computation of $f \cp g$ is $O(f_K + g_K)$.
\end{proposition}

\begin{proof}
  The proof is similar to that of Proposition \ref{proposition:tensor-product-linear},
  except that we must include the cost of the operations $R(f)$ and $L(g)$.
  These operations are linear by proposition \ref{proposition:permutation-linear},
  and so the composition $f \cp g$ must also be $O(f_K + g_K)$.
\end{proof}




\section{Empirical Results}
\label{section:empirical}
We now give an empirical evaluation of our complexity claims on several
synthetic benchmarks.
We compare our own implementation to the wiring diagrams of
\texttt{Catlab.jl}~\cite{Catlab, patterson2020}.\footnote{
  Note however that Catlab's wiring diagrams provide a strictly more general
  setting than ours.
  We discuss possible generalisations of our approach to address this in Section
  \ref{section:discussion}.
}
Both our implementation and benchmarking code are available on GitHub at
\href{https://github.com/statusfailed/cartographer-har}{https://github.com/statusfailed/cartographer-har}.

In the following benchmarks we use a fixed monoidal signature based on the
finite presentation of boolean circuits described in \cite{lafont_towards_2003}.
We choose boolean circuits since they are a real-world application of string
diagrams in which string diagrams would typically be very large.
For example, a string-diagrammatic representation of a CPU would need (at least)
hundreds of thousands of generators.
In particular, our benchmarks use the following generators:
\[ \Sigma = \{ \gcopy : 1 \to 2 \qquad \gxor : 2 \to 1 \qquad \gand : 2 \to 1 \qquad \gnot : 1 \to 1 \} \]

\paragraph{Experiment Details}
Each benchmark has the same structure: for $k \in \{ 1 \ldots 20 \}$ we
construct two string diagrams consisting of $2^{k-1}$ generators, and then measure tensor
product or composition of those diagrams.
We repeat each measurement $10$ times for each $k$ and plot the mean with
minimum and maximum error bars.
Further, if a result takes longer than 60 seconds to compute, it is omitted.
More details of our experimental setup can be found in
the full version of our paper \cite[Appendix C]{coc}.

Note carefully that the performance chart for each benchmark uses a log scale on
both axes, since for each $k$ we construct a string diagram of size $2^k$.\footnote{
  In these experiments, by ``size'' we mean specifically the number of generators in the diagram.
}

\subsection{Benchmark \#1: Repeated Tensor}
\label{section:repeated-tensor}
We first measure the performance of the tensor product of large representations.
Concretely, let $f$ be the $k$-fold tensor product of $\gand$, i.e.,
$f = \gand \otimes \overset{k}\ldots \otimes \gand$.
We measure the performance of computing $f \otimes f$.

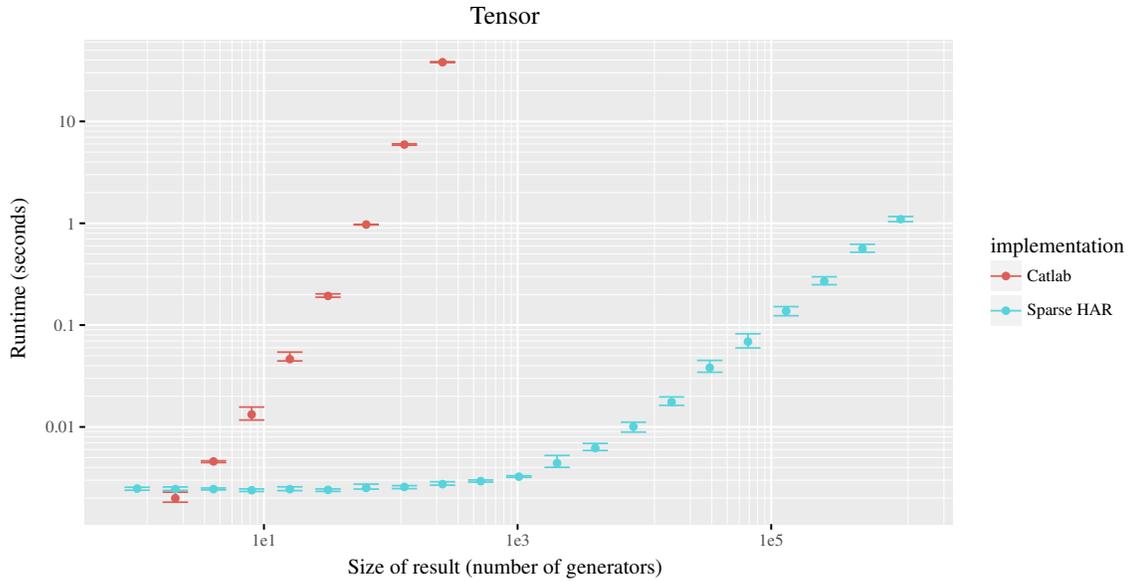
\resizebox{\textwidth}{!}{\input{plots/tensor.pgf}}

\subsection{Benchmark \#2: Small-Boundary Composition}
We measure the performance of composition
$\ObA \overset{f}{\to} \ObB \overset{g}{\to} \ObC$
along a small shared boundary, i.e., where $\ObB \ll f_K + g_K$.
Concretely, let $f$ be the $k$-fold composition of $\gnot$,
so that $f = \gnot \cp \overset{k}\ldots \cp \gnot$.
We measure the performance of computing $f \cp f$.

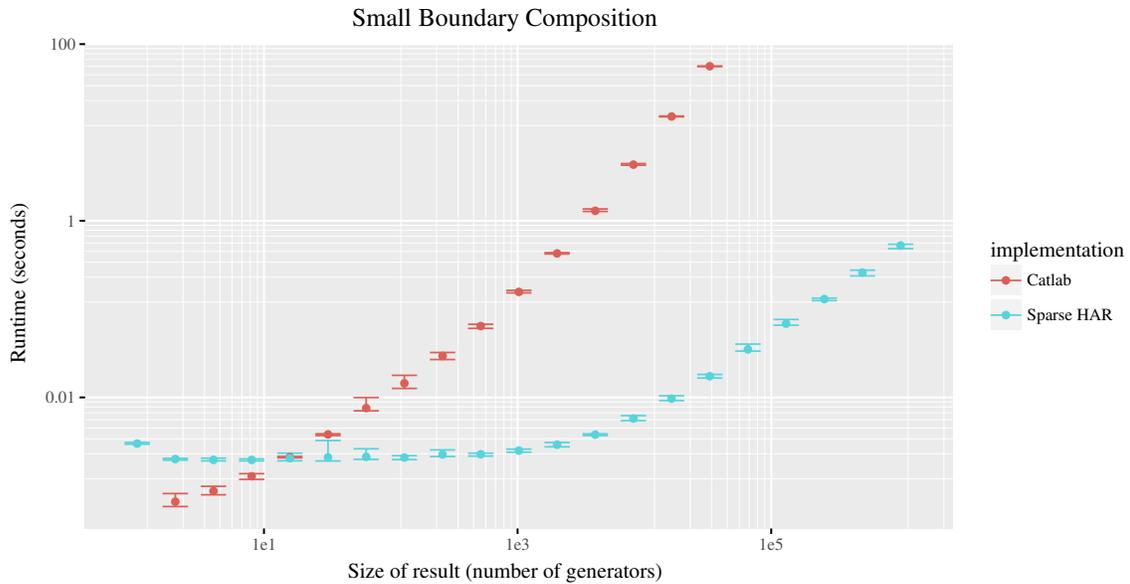
\resizebox{\textwidth}{!}{\input{plots/compose_small.pgf}}

\subsection{Benchmark \#3: Large-Boundary Composition}
We measure the performance of composition
$\ObA \overset{f}{\to} \ObB \overset{g}{\to} \ObC$
along a \emph{large} shared boundary, i.e.\ where
$\ObB \approx \mathsf{min}(f_K, g_K)$.
In particular, let $f$ be the $k$-fold tensor product of $\gnot$.
Then we measure $f \cp f$.

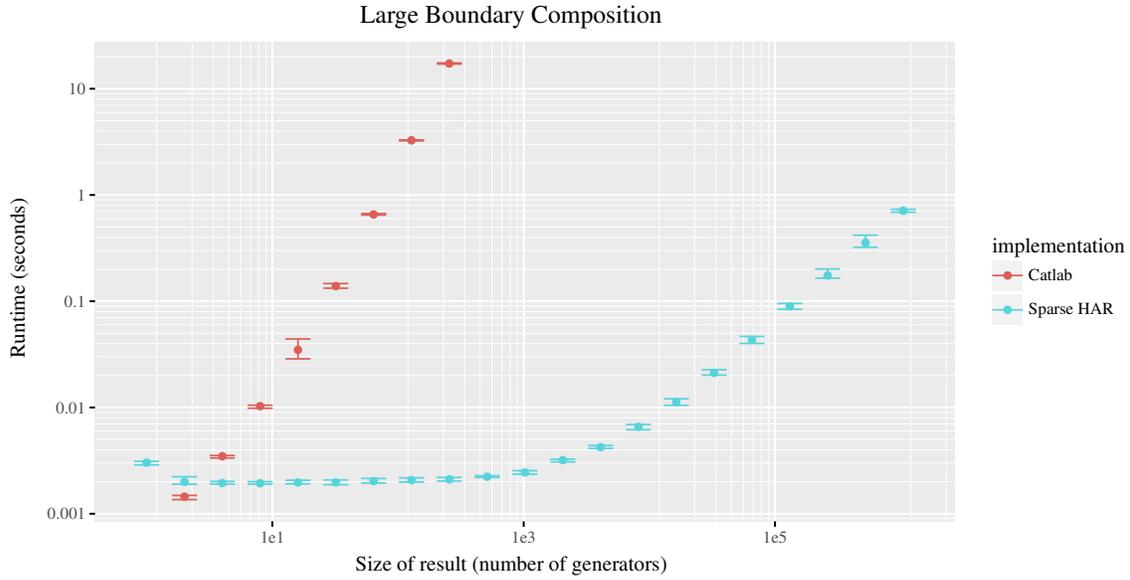
\resizebox{\textwidth}{!}{\input{plots/compose_large.pgf}}

\subsection{Benchmark \#4: Synthetic Benchmark}
We give a final benchmark as a validity check to ensure our implementation still
performs well on realistic-looking representations.
Specifically, we measure the performance of composing two $2^{k-1}$-bit adder
circuits to form a $2^k$-bit adder.

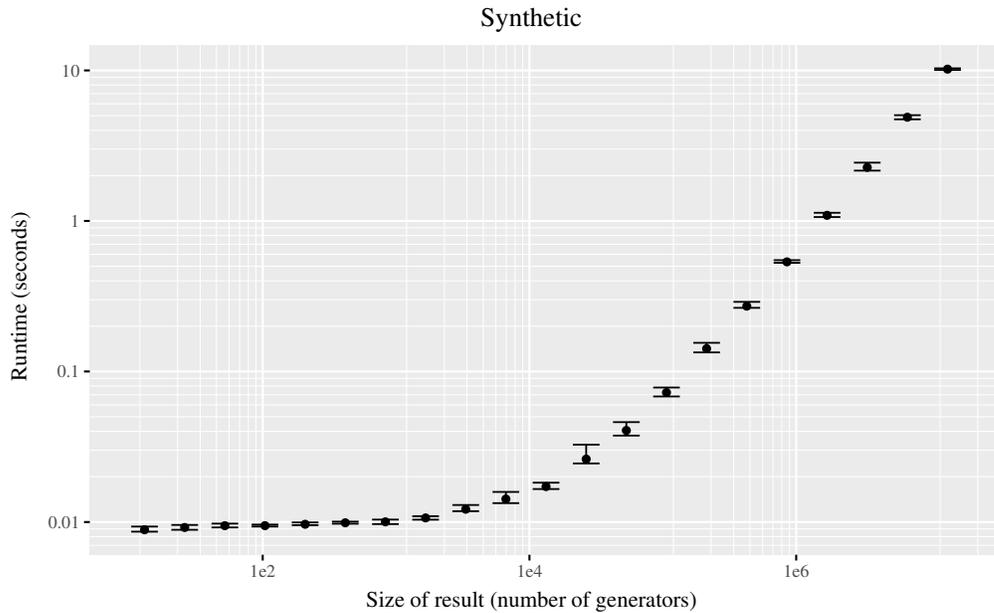
\scalebox{0.85}{\resizebox{\textwidth}{!}{\input{plots/synthetic.pgf}}}

\section{Discussion \& Future Work}
\label{section:discussion}
We consider our work a step towards a set of high-performance algorithms for
manipulating string diagrams, but naturally a number of avenues for improvement
remain.

Most obviously, it remains to explore algorithms for matching and rewriting, which are
necessary to support applications like a string-diagrammatic proof assistant.
Perhaps less obviously, we would also like to study algorithms for \emph{evaluating}
circuit diagrams: this can be useful for e.g., simulating a boolean circuit or writing an
interpreter for a programming language whose syntax is based on SMCs.

There are also several optimizations that could be made to our current algorithm.
Firstly, we represent permutations as matrices, but a more efficient approach
could be to use dense vectors of indices.
However, this would require the implementor to have access to a function like
the \Halfperm{} algorithm of \cite{GustavsonSparse}.

Finally, several generalisations may be possible.
Most useful would be to generalise to arbitrary symmetric monoidal syntax
rather than just PROPs.
Secondly, by modifying our representation slightly, we could
account for arbitrary hypergraphs with interfaces --- although we also believe this
would affect the complexity bounds.





\paragraph{Acknowledgements}
Fabio Zanasi acknowledges support from \textsc{epsrc} EP/V002376/1.

\nocite{*}
\bibliographystyle{eptcs}
\bibliography{main}





\end{document}